\documentclass[leqno,12pt]{article}

\usepackage{epsfig}

\usepackage{amsmath}
\usepackage{amscd}
\usepackage{amsopn}
\usepackage{amsthm}
\usepackage{amsfonts,amssymb}\usepackage{amsfonts,bbm}
\usepackage{latexsym}

\usepackage{color}

\makeatletter
\@addtoreset{equation}{section}
\makeatother

\newtheorem{lemma}{LEMMA}[section]
\newtheorem{proposition}[lemma]{PROPOSITION}
\newtheorem{corollary}[lemma]{COROLLARY}
\newtheorem{theorem}[lemma]{THEOREM}
\newtheorem{remark}[lemma]{REMARK}
\newtheorem{remarks}[lemma]{REMARKS}

\newtheorem{examples}[lemma]{EXAMPLES}

\newtheorem{assumption}[lemma]{ASSUMPTION}
\newcommand{\real}{\mathbbm{R}}

\newcommand{\nat}{\mathbbm{N}}

\newcommand{\ganz}{\mathbbm{Z}}

\renewcommand{\a}{\alpha}
\renewcommand{\b}{\beta}

\newcommand{\g}{\gamma}

\newcommand{\ve}{\varepsilon}

\newcommand{\reald}{{\real^d}}

\newcommand{\on}{\quad\text{ on }}
\newcommand{\und}{\quad\mbox{ and }\quad}
\newcommand{\inv}{^{-1}}
\newcommand{\ov}{\overline}

\newcommand{\W}{\mathcal W}  
\newcommand{\C}{\mathcal C}

\renewcommand{\H}{{\mathcal H}}
\newcommand{\B}{\mathcal B}

\newcommand{\M}{\mathcal M}

\newcommand{\itemframe}%
{\setlength{\parskip}{10pt}\begin{enumerate} \setlength{\topsep}{10pt}%
\setlength{\itemsep}{15pt}\setlength{\parsep}{5pt}}

\newcommand{\vy}{\ve_y}
\newcommand{\vx}{\ve_x}

\newcommand{\Px}{\mathcal P} % P(X)}

\newcommand{\uc}{{U^c}}

\newcommand{\vc}{{V^c}}

\newcommand{\kap}{\operatorname*{cap}}
\newcommand{\tkap}{\widetilde{\operatorname*{cap}}}
\newcommand{\kapi}{\operatorname*{{cap_\ast}}}
\newcommand{\kapo}{\operatorname*{{cap^\ast}}}

\newcommand{\es}{E_{\mathbbm P}}

\title{Harnack inequalities for Hunt processes\\ with  Green
  function\\
}
\author{Wolfhard Hansen and Ivan Netuka}
\date{}

\begin{document}
\maketitle

\begin{abstract}
Let $(X,\mathcal W)$ be a balayage space, $1\in \mathcal W$, or -- equivalently -- let $\mathcal W$
be the set of excessive functions of a Hunt process on a locally compact space~$X$ with countable base such that
$\mathcal W$  separates points, every function in $\mathcal W$ is the supremum of its
continuous minorants and there exist  strictly positive continuous       $u,v\in \mathcal W$ 
such that $u/v\to 0$  at infinity. We suppose that there is a~Green function 
$G>0$ for $X$, a metric $\rho$ on $X$ and a   decreasing function $g\colon[0,\infty)\to (0,\infty]$ 
having the doubling property and a mild upper decay  near~$0$  such that $G\approx g\circ\rho$
(which is equivalent to a $3G$-inequality).

Then the corresponding capacity for balls of radius $r$ 
 is bounded by a~constant multiple of 
$1/g(r)$. Assuming that reverse inequalities hold as well and that
jumps of the process, when starting at neighboring points, are related
in a suitable way, it is proven that positive harmonic functions satisfy
scaling invariant Harnack inequalities. Provided that the
Ikeda-Watanabe formula holds, sufficient conditions for this relation
are given. This shows that rather general L\'evy processes are covered
by this approach.

 {

  }
\end{abstract}

\section{Setting and main result}
Our basic setting will be almost as in \cite{H-liouville-wiener}, but assuming that points 
are polar:

Let $X$ be a locally compact space with countable base.
Let $\C(X)$ denote the set of all continuous real functions on $X$
and let $\B(X)$ be the set of all Borel measurable numerical functions on $X$.
The set of all (positive) Radon measures on $X$ will be denoted by $\M(X)$.

Moreover, let $\W$ be a convex cone of positive lower 
semicontinuous numerical functions on~$X$ such that $1\in \W$ and $(X,\W)$ is a balayage space
(see \cite{BH}, \cite{H-course}  or \cite[Appendix]{HN-unavoidable}). In particular, the following holds: 
 
\begin{itemize} 
\item[\rm (C)]
 $\W$ separates the points of $X$, for every $w\in \W$,
\[
              w=\sup\{v\in\W\cap \C(X)\colon v\le w\}, 
\]
and there are strictly positive $u,v\in\W\cap \C(X)$ such that $u/v\to 0$  at~infinity.   
\end{itemize} 
Then there exists a Hunt process $\mathfrak X$ 
on $X$ such that $\W$ is the set $\es$ of excessive functions for the transition semigroup 
$\mathbbm P=(P_t)_{t>0}$ of $\mathfrak X$ (see \cite[IV.7.6]{BH}), that is,
\[
             \W=\{v\in \B^+(X)\colon \sup\nolimits_{t>0} P_tv=v\}.
\]

We note that, conversely, given any sub-Markov right-continuous semigroup $\mathbbm P=(P_t)_{t>0}$  on $X$
such that (C) is satisfied by its convex cone $\es$ of excessive functions, $(X,\es)$ is a~balayage space, 
and $\mathbbm P$ is the transition semigroup of a Hunt process (see  \cite[Corollary 2.3.8]{H-course}
or \cite[Corollary A.5]{HN-unavoidable}).                 
 
For every subset $A$ of $X$, we have reduced functions  $R_u^A$, $u\in \W$, and reduced measures
$\vx^A$, $x\in X$, defined by
\[
          R_u^A:=\inf\{ v\in \W\colon v\ge u\mbox{ on } A\} \und  \int u\,d\vx^A=R_u^A(x).
\]
Of course,  $R_u^A\le u$ on $X$ and $R_u^A=u$ on $A$.  The greatest lower semicontinuous minorant~$\hat R_1^A$  of $R_1^A$
(which is also the greatest finely lower semicontinuous minorant of~$R_1^A$) is contained 
in~$\W$, and  $\hat R_1^A=R_1^A$ on~$A^c$ (see \cite[VI.2.3]{BH}). If $A$ is not thin at any of its points
 (see \cite[VI.4]{BH})  for the definition), in particular, if $A$ is  open, then $R_u^A\in \W$.
If $A$ is  Borel measurable, then  
\begin{equation}\label{connection}
               R_1^A(x)= P^x[T_A<\infty], \qquad x\in X,
\end{equation} 
where  $T_A(\omega):=\inf \{t\ge 0\colon X_t(\omega)\in A\}$
 (see \cite[VI.3.14]{BH}) and, for every Borel measurable set $B$ in
 $X$, 
\[
    \vx^A(B)=P^x[X_{T_A}\in B; T_A<\infty].
\]

For every open set $U$ in $X$, let $\H^+(U)$ denote the set of all 
functions $h\in \B^+(X)$ which are \emph{harmonic on~$U$} (in the sense of \cite{BH}), that is,
such that $h|_U\in \C(U) $  and
\begin{equation}\label{mv}
    H_V h(x):=\vx^\vc(h):=\int h\,d\vx^\vc=h(x) \qquad
\mbox{ if  }  V\mbox{ is open and } x\in V\subset
\subset U.
\end{equation} 
If, for example, $A$ is a Borel measurable set in $X$ and $u\in \W$,
then,  by \cite[VI.2.6]{BH},   
\begin{equation}\label{A-harmonic}
     R_u^A\in \H^+(X\setminus \ov A) \qquad \mbox{ provided }  u\le w\mbox{ for some } w\in \W\cap \C(X).
\end{equation}

We note that $U\mapsto \H^+(U)$ has the following  sheaf property: If $U_i$, $i\in I$, are open sets in $X$, then
\[ 
                             \bigcap\nolimits_{i\in I}  \H^+(U_i)= \H^+\bigl(\bigcup\nolimits_{i\in I} U_i\bigr).
\]
In fact, given an open set $U$ in $X$, a function $h\in \B^+(X)$ which is continuous on $U$ is already contained
in $\H^+(U)$, if, for every $x\in U$, there exists a fundamental system of relatively compact open neighborhoods $V$ of $x$ 
in $U$ such that $\vx^\vc(h)=h(x)$ (see \cite[III.4.4 and III.4.5]{BH} or \cite[Corollary 5.2.8 and Corollary 5.2.9]{H-course}).

Moreover, let $\tilde \H^+(U)$ be the set of all $h\in\B^+(X)$ such
that $h<\infty$ on $U$ and (\ref{mv}) holds.  Then every function in $\tilde \H^+(U)$ is 
 lower semicontinuous on $U$, and 
\begin{equation}\label{hht}
         \tilde \H_b^+(U)=\H_b^+(U).
\end{equation} 
Indeed, let $V$ be an open set such that $V\subset \subset U$. By (\ref{A-harmonic}), 
$H_V1=R_1^\vc$ is harmonic on~$V$.  
 Moreover, for every $f\in \B_b^+(X)$ with compact support, 
the function~$H_Vf$ is continuous on $V$ (see \cite[III.2.8]{BH}). So, for every $f\in \B^+(X)$,
the function $H_V f$ is lower semicontinuous. Assuming that $f$ is bounded, 
both $H_Vf$ and $H_V(\|f\|-f)$ are lower semicontinuous on $V$, and hence (due to the continuity of the sum)
both are continuous on $V$.

For our main result we shall assume     
that we have a metric $\rho$ for $X$, a~Green function $G$ on $X$, a decreasing function $g$ on $[0,\infty]$,
 and $0<R_0\le \infty$ such that 
\[
G\approx g\circ \rho,
\] 
 $g$ having the  doubling property and  weak upper decay on  $(0,R_0)$ (Assumption~\ref{main-ass}).
Defining balls $B(x,r):=\{y\in X\colon \rho(y,x)<r\}$,  $x\in X$, $0<r<R_0$, we suppose that, for some $c_0\ge 1$, 
the corresponding capacities satisfy
\begin{equation}\label{ball-cap}
 \kap B(x,r)\ge c_0\inv  g(r)\inv
\end{equation} 
(Assumption \ref{ex-c0}; a reverse estimate is a consequence of the previous assumption).
Finally, we shall suppose (Assumption \ref{final-ass}) that there are constants $0<R_1\le \infty$, $c_J\ge 1$ and $0<\a_0<1$ such that,
 for all $x\in X$, $0<r<R_1$ and $y\in B(x,\a_0 r)$, 
\begin{equation}\label{basic-jump}
 \vx^{B(x,\a_0 r)^c} \le c_J\,\vy^{B(x,r)^c} \on    B(x,r)^c . 
\end{equation}

Then our main result is the following (see also Remarks \ref{Green-first},6 and \ref{ess-local}).    

\begin{theorem}\label{harnack} 
\begin{itemize}
\item[\rm (1)] 
For every open set $U$ in $X$, $\tilde \H^+(U)=\H^+(U)$.
\item[\rm (2)] 
\emph{Scaling invariant Harnack inequalities}:   
There exists constants~$ \a\in (0,1)$ and $K\in (1,\infty) $    
such that the following holds: For all  $x_0\in X$,
$0<R<R_0\wedge R_1$ 
such that $\ov B(x_0,R)$ is a~proper compact subset of $X$, and all  $h\in \B^+(X)$
which are harmonic in a neighborhood of $\ov B(x_0,R)$,
\begin{equation}\label{har}
     \sup h(B(x_0,\a R))\le K \inf h(B(x_0,\a R)).
\end{equation} 
\end{itemize}
\end{theorem}

In Section 2  we shall shortly discuss a Green function  for $(X,\W)$ and the related capacity.
In Section 3, the probability of  hitting a~subset $A$ of a  ball before leaving a~much larger ball is estimated
in terms of the capacity of~$A$. And in Section~4 we repeat basic facts on the relation between   
a lower estimate of the equilibrium potential of a ball by the Green function and (\ref{ball-cap}).

Having prepared it by two crucial lemmas in Section 5,
the proof of Theorem~\ref{harnack} is given in Section~6. 
Sufficient conditions for the validity of (\ref{basic-jump})  are discussed in  Section 7. 
In a final Section~8, we prove Harnack inequalities
under intrinsic and local assumptions.

\section{Green function and capacity}\label{green-function}

By definition, a \emph{potential on $X$} is a function $p\in \W$ such that, for every relatively compact open 
set $V$ in $X$,
the function $H_Vp=R_p^\vc$ is continuous and real on $V$ (and hence harmonic on $V$) and 
\[
                       \inf \{ R_p^\vc\colon V\mbox{ relatively compact open in }X\}=0.
\]
By \cite[Proposition 4.2.10]{H-course}, a function $p\in \W\cap \C(X)$ 
 is a potential if and only if there exists a strictly positive $q\in \W\cap \C(X)$
such that $p/q$ vanishes at infinity.  Let~$\Px$ denote the set of all continuous real potentials on $X$.

Unless stated otherwise 
 we assume from now on the following:

\begin{assumption}\label{main-ass}
There exists a Borel measurable function $G\colon X\times X\to (0,\infty]$ such that 
$G=\infty$ on the diagonal 
and the following holds:
\begin{itemize} 
\item [\rm (i)] 
   For every $y\in X$,  $ G(\cdot,y)$ is a potential which is harmonic on $X\setminus \{y\}$.   
\item [\rm (ii)] 
For every potential $p$ on $X$, there exists a~measure $\mu$ on $X$    such that 
\begin{equation}\label{p-rep}
p=G\mu:=\int G(\cdot,y)\,d\mu(y).
\end{equation} 
 \item [\rm (iii)]
There exist a metric $\rho$ for~$X$  {\rm(}compatible with the
topology of~$X${\rm)}, a 
decreasing numerical function $g>0$ on $[0,\infty)$, and constants $c\ge 1$, $c_D>1$, $0<\a_0<1$, $0<\eta_0<1$, 
$0<R_0\le \infty$ such that  
\begin{equation}\label{Gg}
c\inv g\circ \rho\,  \le\,  G \, \le\,  c\, g\circ \rho,
\end{equation} 
and, for every $0<r<R_0$, 
\begin{equation}\label{doubling}
g(r/2)\le c_D g(r) \und                   g(r)\le \eta_0 g(\a_0 r).
\end{equation} 
\end{itemize} 
\end{assumption}

\begin{remarks}\label{Green-first}
{\rm
1. Having (i), each of the following properties implies (ii).
\begin{itemize}
\item
$G$ is lower semicontinuous on $X\times X$, continuous outside the diagonal,
the potential kernel $V_0:=\int_0^\infty P_t\,dt$ of $\mathfrak X$ is 
proper, and there is a measure~$\mu$ on $X$ such that $V_0f =\int G(\cdot, y)f(y)\,d\mu(y)$, $f\in\B^+(X)$
(see \cite{maagli-87}  and \cite[III.6.6]{BH}).

\item
$G$ is locally bounded off the diagonal, each function $G(x,\cdot)$ is lower semicontinuous 
on $X$ and continuous on $X\setminus \{x\}$, and there exists a measure $\nu$ on $X$ such that $G\nu\in \C(X)$ 
and $\nu(U)>0$, for every  finely open  $U\ne \emptyset$ (the latter holds, for example, if 
$V_0(x,\cdot)\ll \nu$, $x\in X$). See \cite[Theorem 4.1]{HN-representation}.
\end{itemize} 

2. 
The measure in (\ref{p-rep}) is uniquely determined and, given any measure~$\mu$
on~$X$ such that $p:=G\mu$ is a potential, the complement of the support of $\mu$ is the largest open
set, where $p$ is harmonic (see, for example, \cite[Proposition 5.2 and Lemma~2.1]{HN-representation}).

3. For the special case $X=\reald$ with $\rho(x,y)=|x-y|$ and isotropic unimodular Green function,
covering rather general L\'evy processes, see \cite[Section 6]{HN-unavoidable} and \cite{H-fuku}.

4. Of course, (\ref{doubling}) implies that, for any $\eta>0$, there exists $\a\in (0,1)$ such that 
$g(r)\le \eta g(\a r)$ for every $0<r<R_0$ (it suffices to choose $m\in\nat$ such that $\eta_0^m\le \eta$ 
and to take $\a:=\a_0^m$). 
Moreover, we see that $\lim_{r\to 0} g(r)=\infty$.

5. We may (and shall)  
assume without loss of generality that $g$ is continuous on $(0,R_0)$.  
 Indeed, if $R_0=\infty$, let $Z:=\{2^n\colon n\in\ganz\}$ and let $\tilde g$ be the continuous function 
on $(0,\infty)$ such that  $\tilde g=g$ on $Z$ and $\tilde g$ is locally affinely linear on $(0,\infty)\setminus Z$. 
Then $g$ is decreasing, and it
is easily verified that $c_D\inv g\le \tilde g\le c_D g$.

 If $R_0<\infty$, we may proceed similarly using 
$Z:=\{2^{-n}R_0\colon  n\in \nat\}$ and defining $\tilde g(R_0):=\lim_{s\to R_0-} g(s)$.

6. Let us mention that it is rather easy and straightforward to show the following.     
If all balls are relatively compact and the doubling property $g(r/2)\le c_D g(r)$ holds for all
$0<r<\infty$, then 
\[
  \sup v(B(x_0,R/2))\le (cc_D)^2 \inf v(B(x_0,R/2)) 
\] 
for all $v\in \W$, $x_0\in X$ and $R>0$ such that $v$  is contained in $\tilde \H^+(U)$ for some open
neighborhood $U$ of $\ov B(x_0,R)$. 
If, in addition, the   function $1$ is harmonic on $X$, then
every function in $\tilde \H^+(X)$ is constant (Liouville property; see \cite[Section 2]{H-liouville-wiener}).
}
\end{remarks}

Suppose that $A$ is a subset of  $X$ such that $\hat R_1^A$ is a potential. 
Then there is a~unique measure~$\mu_A$ on $X$, 
the \emph{equilibrium measure for $A$}, such that 
\[
               \hat R_1^A=G\mu_A.
\]           
If $A$ is open, then $\hat R_1^A=R_1^A\in \H(X\setminus \ov A)$
and  $\mu_A$ is supported by~$\ov A$. For a general balayage space this may already fail if $A$ is compact (see \cite[V.9.1]{BH}).

We  define inner capacities for open sets $U$ in $X$ by 
\begin{equation}\label{inner-cap}
 \kapi U:=\sup\bigl\{\|\mu\|\colon  \mu\in\M(X), \ \mu(X\setminus U)=0,\ G\mu\le 1\bigr\}
\end{equation} 
and outer capacities for arbitrary sets $A$ in $X$ by
\begin{equation}\label{outer-cap}
 \kapo A:=\inf\bigl\{ \kapi U\colon U\mbox{ open neighborhood of } A\bigr\}.
\end{equation} 
Obviously, $\kapo A=\kapi A$, if $A$ is open. If $\kapi A=\kapo A$, we might simply write $\kap A$
and speak of the capacity of $A$. It is easily seen that $U\mapsto \kap U$ is subadditive and $\kap U_n\uparrow \kap U$,  
for any sequence $(U_n)$ of open sets in $X$ with $U_n\uparrow U$.

The capacity of open sets $U$ is essentially determined by the total mass of
equilibrium measures for  relatively compact open sets in $U$ (see \cite[Lemma 1.6]{H-liouville-wiener}):

\begin{lemma}\label{cap-potential}
For every open set $U$ in $X$,
\[
\kap U \ge \sup\{\|\mu_V\|\colon V\mbox{ open and }\ov V \mbox{ compact in } U\}\ge 
c^{-2} \kap U.
\] 
\end{lemma}

\section{Hitting of sets before leaving large balls} 

Let us first recall the following simple fact (see \cite{H-fuku}), where,
as usual,    
\[
       \tau_U:=T_{\uc}.
\]

\begin{lemma}\label{subset-U}
Let  $A$ be a Borel measurable set in an open set $U\subset X$ and $\g>0$.\\
If  $R_1^A\le \g$ on $U^c$, then \\[-3mm]
\[
P^x[T_A<\tau_U]\ge R_1^A(x)- \g,\qquad\mbox{ for every } x\in U.
\]
\end{lemma}

Using Lemma \ref{cap-potential}, this   leads to a lower estimate 
for the probability of hitting a~subset of a~ball before leaving a much larger ball (cf.\  \cite[Proposition 4]{H-fuku}).

\begin{proposition}\label{hit-A}
Let $\eta:= (2c^3c_D^2)\inv$ and let  $0<\a<1/2$, $0<r<R_0$ be  such that   $g((1-2\a) r)\le
c\eta g(\a r)$. Then, for all $x_0\in X$,  
$x\in B:=B(x_0,2\a r)$, and Borel measurable sets $A$ in~$B(x_0,2\a r)$,
\begin{equation}\label{hitting-est}
    P^x[T_A<\tau_{B(x_0,  r)} ]\ge \eta  g(\a r) \kapo(A).
\end{equation} 
\end{proposition} 

\begin{proof} To prove  (\ref{hitting-est}) we may assume without loss of
generality that $A$ is open (see \cite[VI.3.14]{BH}). Let~$V$ be an open set such that $\ov V$~is compact in $A$.
Since $\rho(x, \cdot) \le 4\a r$ on $\ov V$, we have
\[
        R_1^V(x)=\int G(x,z)\,d\mu_V(z)\ge   c\inv g(4 \a r) \|\mu_V\|\ge  2 c^2\eta g(\a r) \|\mu_V\| .
\]
If  $y\in X\setminus B(x_0, r)$,  then $\rho(y,\cdot)\ge (1-2\a )r$ on $\ov V$,   and therefore
\[
  R_1^V(y)=\int G(y,z)\,d\mu_V(z) \le c g((1-2\a) r) \|\mu_V\|\le c^2 \eta g(\a r)  \|\mu_V\|.
\]
So, using Lemma \ref{subset-U}, 
\[
P^x[T_A<\tau_{B(x_0, r)}] \ge P^x[T_{V}<\tau_{B(x_0,  r)}]   \ge  c^2\eta  g(\a r) \|\mu_V\|.
\] 
An application of Lemma \ref{cap-potential} completes the proof.
\end{proof}

\begin{remark}{\rm
Let us note that our probabilistic statements and proofs can be replaced by analytic ones
using that, for all Borel measurable sets $A,B$ in an open set $U$,
\[
       P^x[X_{T_A}\in B; T_A<\tau_U] =\vx^{A\cup \uc}(B)
\]
(see \cite[VI.2.9]{BH}) and, for all Borel measurable sets $B$ in $X$ and $B\subset A\subset X$,
\begin{equation}\label{it-bal}
     \vx^B=\vx^A|_B+ \bigl(\vx^A|_{B^c}\bigr){}^B.
\end{equation} 
(If $x\in B$,  then (\ref{it-bal}) holds trivially. If $x\notin B$ and $p\in \Px$,  then, by  \cite[VI.9.1]{BH},
\[
\hat R_p^B(x)= R_p^B(x)=\int R_p^B\,d\vx^A=\int_B p\,d\vx^A+\int_{B^c} \hat R_p^B\,d\vx^A.)
\]
}
\end{remark}

\section{Equilibrium potential and capacity of balls}

The following estimates are 
\cite[Proposition 1.7]{H-liouville-wiener}.

\begin{proposition}\label{Rge}
Let $x\in X$, $r>0$, and $B:=B(x,r)$. Then
the reduced function  $R_1^B$ is a potential {\rm(}in fact, bounded by a potential $p\in \Px${\rm)}, 
\[
   R_1^B\le c \,\frac{G(\cdot,x)}{g(r)}\le c^2\, \frac{ g(\rho(\cdot,x))}{g(r)}, \qquad \|\mu_B\|\vee \kap B\le c\,\frac 1{g(r)}, 
\]
\[
 R_1^B\ge  c\inv  \kap B \cdot g(\rho(\cdot,x)+r).
\]
\end{proposition}

For the next three sections  
we assume, in addition,  the following. 

\begin{assumption}\label{ex-c0}
 There exists $c_0\ge 1$  such that,  for all $x\in X$ and  $0<r<R_0$,
\begin{equation}\label{cap-R}
                                          \kap B(x,r)\ge c_0\inv  g(r) \inv. 
\end{equation} 
\end{assumption}

Then,  by Proposition   \ref{Rge},  for all $x\in X$ and  $0<r<R_0$,    
\begin{equation}\label{RB}
                 R_1^{B(x,r)} \ge (cc_0)\inv \, \frac{g(\rho(\cdot,x)+r)}{g(r)}\,. 
\end{equation}  

\begin{examples}\label{diff} {\rm
1. Assume for the moment that  $(X,\W)$ is a harmonic space, that is, $\mathfrak X$ is a diffusion.
Moreover, suppose that $X$ is non-compact, but balls are relatively
compact. Then (\ref{cap-R}) holds with $c_0:=c^3c_D$ if $0<r<R_0/2$.  

Indeed, let $x\in X$, $0< r< R_0/2$, $B:=B(x,r)$.    
Given $\tilde c>c$,  we have $G(\cdot,x)\le c g(r)<\tilde c g(r)$  on $X\setminus   B$, hence, for some $0<\tilde r<r$,
 $G(\cdot,x)< \tilde c g(r)$  on~$X\setminus B(x,\tilde r)$,  
and we see,  by the minimum principle  (\cite [III.6.6]{BH}), 
that $G(\cdot,x)\le \tilde c g(r) R_1^B$ on~$X\setminus B(x,\tilde r)$. So $G(\cdot,x)\le  c g(r) R_1^B$ on~$X\setminus B$.

Choosing $y\in X\setminus \ov B(x,2r)$, we know that the potential $G(\cdot,y)$ is strictly positive and harmonic on $B(x,2r)$,
hence  $\partial B(x,2r)\ne \emptyset$. Let  $z\in \partial B(x,2r)$. Then
\[
R_1^B(z)\ge c\inv  G(z,x)/g(r) \ge c^{-2} g(2r)/g(r)\ge (c^{2}c_D)\inv .
\]
Let $a<R_1^B(z)$. By \cite[VI.1.2]{BH}, there exists   $0<s<r$ such that $V:=B(x,s)$ satisfies $a<R_1^V(z) $.
Since  $\rho(z,\cdot)>r$ on $B$,
\[
      R_1^V (z) =\int G(z,y)\,d\mu_V (y)\le  c g(r)  \|\mu_V\|\le c g(r)\kap B.
\] 
Thus $(c^3c_D)\inv  g(r)\inv \le \kap B$.

2. If $X=\reald$ and $\rho(x,y)=|x-y|$,   
 then Assumption \ref{ex-c0} is satisfied provided 
there exists $C_G\ge 1$ such that      
$d\int_0^r s^{d-1} g(s)\,ds\le C_G r^d g(r)$ for all $0<r< R_0$, since then the normalized Lebesgue measure
$\lambda_{B(x,r)}$ on $B(x,r)$ satisfies   
$G\lambda_{B(x,r)} \le G\lambda_{B(x,r)} (x)\le c C_G g(r)$ (see~\cite{H-fuku}).
So Assumption \ref{ex-c0} is satisfied for rather general isotropic unimodular L\'evy processes. 
}
\end{examples} 

\section{Two crucial lemmas}\label{crucial}

In this and the following section, we assume the following on the jumps.

\begin{assumption}\label{final-ass} 
There exist $0<R_1\le \infty$, 
$c_J>0$ and  $0<\a<1/2$   
 such that, for all $x\in X$, $0<r<R_1$ and $y\in B(x,\a r)$, 
\begin{equation}\label{jump}
      \vx^{B(x,\a r)^c}\le c_J \, \vy^{B(y,  r)^c} \on   B(y,  r)^c.
\end{equation} 
\end{assumption}

\begin{remarks} \label{ass-remarks}
{\rm
1.  If $\mathfrak X$ is a diffusion or -- equivalently -- if $(X,\W)$
is a harmonic space, then  Assumption \ref{final-ass}  holds trivially, since   
the measures  $\vx^{B(x,\a r)^c}$ do not charge the complement of $  B(y,  r)$.

2. If $0<\a'\le \a$, then $B(x, \a r)^c\subset B(x, \a' r)^c$, and hence, by (\ref{it-bal}),   
\begin{equation}\label{MM'}
 \vx^{B(x, \a' r)^c}|_{B(x, \a r)^c} \le \vx^{B(x, \a r)^c}.
\end{equation} 
Therefore we may replace $\a$ in (\ref{jump}) by any smaller $\a'$. 

3. Similarly, (\ref{it-bal}) implies that, for every $y\in B(x,\a r)$, $\vy^{B(y,r)^c}|_{B(x,2r)^c}\le
\vy^{B(x,2r)^c}$  and $\vy^{B(x,r/2)^c}|_{B(y,r)^c}\le \vy^{B(y,r)^c}$.
Hence 
Assumption \ref{final-ass} is equivalent to the assumption, where 
(\ref{jump}) is replaced by 
\begin{equation}\label{jump-1}
  \vx^{B(x,\a r)^c}\le c_J \, \vy^{B(x,  r)^c} \on   B(x,  r)^c.
\end{equation} 
}
\end{remarks}

For a proof of Theorem \ref{harnack}, 
 we  employ essential ideas from \cite{bass-levin}. However,  not
 assuming the existence of a volume measure
 and not having any information on the expectation of hitting times,   
 we shall rely entirely on capacities of sets. 

A very similar approach
has been used in \cite{mimica-harnack}, where the L\'evy process on $\reald$,
$d\ge 3$, with characteristic exponent $\phi(\xi)=|\xi|^2
\ln^{-1}(1+|\xi|^2)-1$ is considered, and $g(r)\approx
r^{2-d}\ln(1/r)$ as $r\to 0$.     

As in Proposition \ref{hit-A}, let $\eta:=(2c^3 c_D^2)\inv$. We may choose $0<\a<1/4$ such that 
Assumption \ref{final-ass} holds with (\ref{jump-1}) in place of (\ref{jump})  and 
\begin{equation}\label{a-choice}
g(  r)\le cc_D\inv \eta g( \a r)\qquad\mbox{ for every }0<r<R_0.
\end{equation} 
 Since $1-2\a\ge 1/2$, we know that
$g((1-2\a)r)\le c_D g(r)\le c \eta g(\a r)$ for every $0<r<R_0$. 

Moreover, let
\[
        \b:=\frac \eta{6 c_0}, \quad \g:=\frac
         16\wedge \frac \b{c_J},\qquad \kappa:=3\b\g=\frac{\eta\gamma}{2 c_0}.
\]
We choose $j_0,m_0,m_1 \in \nat$ such that  
\[
(1+\b)^{j_0}>c_D, \qquad 2^{m_0}>2j_0,\qquad  2^{m_1}\a^2>1,
\]
and define  
\begin{equation}\label{def-K}
           K:=\kappa\inv c_D^{m_0+m_1}   .   
\end{equation} 

Now we fix $x_0\in X$ and $0<R<R_0\wedge R_1$ 
 such that $ B(x_0,R)$ is    
relatively compact. Since $\lim_{r\to 0} g(r)=\infty$   and $g$ is decreasing and continuous on $(0,R_0)$, 
 we may choose $r_j>0$,
  such that 
\begin{equation}\label{choice-rj}
 g( r_j)   = c_D^{m_0}    (1+\b)^{j-1} g(\a^4 R), \qquad j\in\nat.
\end{equation} 

The following two lemmas are crucial for the proof of Theorem
\ref{harnack}.  

\begin{lemma}\label{sum-rj}
The sum of all $r_j$, $j\in\nat$, is less than $\a^4 R$.
\end{lemma} 

\begin{proof} If $1\le k\le j_0$ and $m\ge 0$, then
$g(r_{mj_0+k})> c_D^{m_0+m} g(\a^4R)\ge g(2^{-(m_0+m)} \a^4R)$,
 and hence $r_{mj_0+k}<2^{-(m_0+m)}\a^4 R$. Thus
\[
\sum\nolimits_{j=1}^\infty r_j
< \sum\nolimits_{m=0}^\infty j_0 2^{-(m_0+m)} \a^4R =j_0 2^{-m_0+1}\a^4R<\a^4 R.
\]
\end{proof}

\begin{lemma}\label{main-lemma}
Let $h\in\H_b^+(B(x_0,R))$ such that $h(y_0)=1$ for some $y_0\in B(x_0,\a^2 R)$. 
If~$j\in\nat$ and  $x\in B(x_0,2\a^2 R)$   such that 
\[ 
h(x)>(1+\b)^{j-1}K,
\]
then there exists $x'\in B(x, r_j/\a^2)$ such that 
\[
       h(x')>(1+\b)^j K.
\]
\end{lemma} 

\begin{proof} Let $j\in\nat$, $r:=r_j$,  and 
$x\in B(x_0,2\a^2R)$  with $h(x)>(1+\b)^{j-1}K$.
Let 
\[
     U_1:=B(x,  r)\cap \{h>\g h(x)\} \und  U_2:=B(x,  r)\cap \{h<2\g h(x)\}. 
\]
Then $U_1$ and $U_2$ are open sets and  $U_1\cup U_2=B(x,   r)\subset B(x_0,2\a^2  R)$. In particular,  
\begin{equation}\label{appl-subadd}
\kap B(x,  r)\le \kap U_1+\kap U_2.
\end{equation} 

 If  $V$ is an open set with $ \ov V\subset U_1$, then, 
by Proposition \ref{hit-A}, 
\begin{eqnarray*} 
  1=h(y_0)&=&\ve_{y_0}^{\ov V\cup B(x_0,\a R)^c}(h)\ge 
  \g h(x) \ve_{y_0}^{\ov V\cup B(x_0,\a R)^c}(\ov V)\\[1.5mm]
&\ge& \g h(x) P^{y_0} [T_V< \tau_{B(x_0,\a R)} ] 
\ge \eta \g h(x) g(\a^2 R)\kap V.
\end{eqnarray*} 
So $\kap U_1\le (\eta \g h(x) g(\a^2 R))\inv$. 
By Assumption \ref{ex-c0},  $\kap B(x,r)\ge (c_0g(r))\inv$. 
Since $g(\a^4R)\le g(2^{-m_1} \a^2 R) \le c_D^{m_1} g(\a^2 R)$,    
we conclude,  by  (\ref{choice-rj}) and (\ref{def-K}), that   
\[
\frac {\kap U_1}{\kap B(x,r)}\le \frac {c_0 g(r)} {\eta\g h(x) g(\a^2 R)}
=\frac{c_D^{m_0} (1+\b)^{j-1} g(\a^4 R)}{2\kappa h(x) g(\a^2R)}\le \frac {(1+\b)^{j-1} K}{2h(x)}<\frac 1{2 }.
\]
By (\ref{appl-subadd}), we obtain that
\begin{equation}\label{cap-V}
 \kap U_2>   (1/2)\kap B(x,r) \ge (2c_0g(r))\inv. 
\end{equation} 
We  choose an open set $W$ such that $\ov W\subset U_2$, $\kap W>
(2 c_0 g(r))\inv$, and define   
\[
        L:=\ov W, \qquad  \nu:=\vx^{L\cup B(x,r/\a)^c}.
\]
Then, by Proposition \ref{hit-A},
\[
\nu(L)= P^x[T_L<\tau_{B(x, r/\a)} ] \ge P^x[T_W<\tau_{B(x,r/\a)} ] \ge \eta g(r)\kap W >   \frac {\eta}{2 c_0}=3\b.
\]

By Lemma \ref{sum-rj}, 
 $r/\a^2< \a^2 R$, and hence  $\ov B(x,r/\a^2)\subset B(x_0,3\a^2R)\subset B(x_0,R)$. 
We   claim that $H:=1_{B(x,r/\a^2)^c} h$ satisfies
\begin{equation}\label{m2c}
\vx^{B(x,r/\a )^c} (H) \le \b h(x).
\end{equation} 
Indeed, if not,  then (\ref{jump-1})  implies that, for every $y\in B(x,r)$, 
\[
h(y)=\vy^{B(x,r/\a^2)^c}(h) =\vy^{B(x,r/\a^2)^c}(H)\ge c_J\inv \vx^{B(x,r/\a)^c}(H)> c_J\inv \b h(x)\ge \g h(x),
\]
contradicting the fact that $U_1$ is a proper subset of $B(x,r)$.

Finally, let $a:=\sup h(B(x, r/\a^2))$. Then
\[
h(x)=\nu(h)\le 2\g h(x)\nu(L)+\int_{X\setminus B(x,r/\a) }h\,d\nu,
\]
where
\[
\int_{B(x, r/\a^2)\setminus B(x,r/\a )} h\,d\nu\le a \nu(B(x, r/\a^2)\setminus B(x,r/\a ))\le a(1-\nu(L))
\]
and, by (\ref{it-bal}) and (\ref{m2c}),  
\[
\int_{B(x, r/\a^2)^c} h\,d\nu=\nu(H) 
\le \vx^{B(x,r/\a)^c}(H)\le \b h(x).
\]
Therefore 
\[
h(x)\le 2\g h(x)\nu(L)+a(1-\nu(L))+\b h(x),
\]
and 
\begin{equation}\label{final-est}
      a \ge \frac{1-\b-2\g \nu(L)}{1-\nu(L)} \, h(x)>  (1+\b)h(x) > (1+\b)^j K
\end{equation} 
completing the proof (since $1-2\g\ge 2/3$ and $\nu(L)>3\b$, we have $ (1-2\g)\nu(L)>2\b$,
hence 
$1-\b-2\g \nu(L)>1+\b-\nu(L)\ge (1+\b)(1-\nu(L))$.
\end{proof} 

Finally, we shall use the following little observation. 

\begin{lemma}\label{positive}
Let $U:=B(x,R)$, $x\in X$, $R>0$,  such that $\ov U$ is a proper compact  subset of X.  
Then there exists a function $h\in \H_b^+(U)$ such that $h>0$ on $\ov U$.
\end{lemma}

\begin{proof} Let $y\in X\setminus \ov U$ 
 and let $V$ be a relatively compact open neighborhood of $\ov U$ such that $y\notin \ov V$.
 Then, for every $n\in\nat$,
\[
           h_n:=H_V (G(\cdot,y)\wedge n) \in \H_b^+(V) \und h_n\uparrow H_VG(\cdot,y)=G(\cdot,y),
\]
as $n\to\infty$. 
Since $G(\cdot,y)>0$, there exists $n\in\nat$ such that $h_n>0$ on $\ov U$. 
\end{proof}

\section{Proof of Theorem \ref{harnack}}

Let us first give a complete statement of Theorem \ref{harnack}.

\begin{theorem}\label{harnack-complete}
Let $(X,\W)$ be a balayage space, $1\in \W$,  suppose that the Assumptions
{\rm \ref{main-ass}, \ref{ex-c0}, \ref{final-ass}} are satisfied and let $\a, K$ be as in 
Section~\ref{crucial}  {\rm(}see {\rm(\ref{a-choice}), (\ref{def-K}))}.    
Then the following hold.

\begin{itemize}
\item[\rm (1)] 
For every open set $U$ in $X$, $\tilde \H^+(U)=\H^+(U)$.
\item[\rm (2)] 
\emph{Scaling invariant Harnack inequalities}: 
Let  $x_0\in X$,  $0<R<R_0\wedge R_1$, and $B:=B(x_0,R)$ such that $\ov B$ is a proper compact subset of $X$.
Then, for all functions  $h\in \B^+(X)$ which are harmonic in a neighborhood of $\ov B$, 
\begin{equation}\label{har-proof}
     \sup h(B(x_0,\a^2 R))\le K \inf h(B(x_0,\a^2 R)).
\end{equation} 
\end{itemize}
\end{theorem} 

\begin{proof} (a) To  prove (2), let $B_0:=B(x_0,\a^2 R)$,    
and let us first consider $h\in \H_b^+(B)$ with $h(y_0)=1$ for
some point $y_0\in B_0$. Then
\begin{equation}\label{b-est}
h\le K \on B_0. 
\end{equation} 
Indeed, suppose that $h(x_1)>K$ for some $x_1\in B_0$.
Then, by Lemmas~\ref{sum-rj} and~\ref{main-lemma}, there exist points
$x_2,x_3,\dots$ in  $B(x_0,2\a^2 R)$ such that
 $h(x_j)>(1+\b)^{j-1} K$, $j\in\nat$.  This contradicts the boundedness of $h$.

 (b)  Next let $h$ be an arbitrary function in $\H_b^+(B)$. 
By Lemma~\ref{positive}, there exists $h_0\in \H_b^+(B)$  such that 
$h_0>0$ on $\ov B$. Let \hbox{$y\in B_0$} and $\ve>0$.
Applying (\ref{b-est}) to $ (h+\ve h_0)/(h+\ve h_0)(y)$ we get that
$h+\ve h_0\le K(h+\ve h_0)(y)$ on $B_0$. Thus (\ref{har-proof}) holds.

(c) Let us now consider 
an open neighborhood $W$ of $\ov B$ and   $h\in \tilde \H^+( W)$. By~ (\ref{hht}),
\begin{equation}\label{hnW} 
h_n:=H_W(h\wedge n)\in \H_b^+(W),
\end{equation} 
$n\in\nat$
(the relation $H_VH_W=H_W$ for open $V\subset \subset W$ is a special case of (\ref{it-bal})).  
By~(b), $\sup h_n(B_0)\le K\inf h_n(B_0)$. Clearly, $h_n\uparrow h$ as $n\to \infty$. So $h$~satisfies
(\ref{har-proof}).

(d)  To prove (1), let $U$ be an   open set in $X$, $h\in\tilde \H^+(U)$,
and $x_0\in U$. We choose $0<R<R_0\wedge R_1$ such that the closure of $ B=B(x_0, R)$
is a proper compact subset of  $U$. 
Let $W$ be a relatively compact open neighborhood of $\ov B$ in $U$ 
and let~$h_n$, $n\in\nat$, be as in (\ref{hnW}).
Then $h-h_n\in \tilde \H^+(W)$  for every $n\in\nat$,  and hence,
by (c), 
\[
         h-h_n \le K (h-h_n)(x_0) \on B_0=B(x_0,\a^2 R).
\]
So the
functions $h_n$ (which are continuous on $W$)  converge to $h$ uniformly on $B_0$.
Therefore $h|_{B_0}\in \C(B_0)$. Thus  $h|_U\in \C(U)$ completing the proof. 
\end{proof}

\begin{remark} \label{ess-local}{\rm
The preceding proofs show that, given $x_0\in X$ and  $R>0$ such that $\ov B(x_0,R)$
is a proper compact subset of $X$, we still obtain (\ref{har-proof}) with some $K\in (1,\infty)$,
which may depend on $x_0$ and $R$, provided  there exist $c_0, c_J\in (0,\infty) $ and $\a\in (0,1)$
such that (\ref{cap-R}) and (\ref{jump})   hold for all $x\in B(x_0,R)$ and $0<r<R$.
For an application see Section \ref{section-local}.
}
\end{remark}

\section{Sufficient conditions for Assumption \ref{final-ass}}\label{suff-ass}

For relatively compact open sets $V$ in $X$,   let $G_V$ denote  the associated Green function on $V$, that is,   
\[
          G_V(\cdot, y) :=G(\cdot ,y)-R_{G(\cdot,y)}^\vc, \qquad y\in V.
\]
We shall need the following simple statement.

\begin{lemma}\label{GGB}
There exists $0<\a<1/4$ such that, for all $y\in X$ and $0<r<R_0$,
\begin{equation}\label{GGM}
                G_{B(y,r)}(\cdot,y)\ge \frac 12 \, G(\cdot,y) \on B(y,2\a r).
\end{equation} 
\end{lemma} 

\begin{proof} Let $0<\a<1/4$ such that $g(r)\le (2c^2c_D)\inv g(\a r)$ for every $0<r<R_0$. 
Let  $y\in X$ and  $0<r<R_0$. Since $G(\cdot,y)\le c g(r) $ on ${B(y,r)^c}$, we obtain that
$R_{G(\cdot,y)}^{B(y,r)^c} \le c g(r)\le (2cc_D)\inv g(\a r) $, 
whereas $G(\cdot,y)\ge c\inv g(2\a r)\ge (c c_D)\inv g(\a r)$ on~$B(y,2\a r)$. So (\ref{GGM}) holds.
\end{proof} 
 
In this section, let us assume the following estimate of Ikeda-Watanabe type,  
which by \cite[Example 1 and Theorem 1]{ikeda-watanabe} holds, with $C_N=1$ and    on $X\setminus \ov B(x,r)$,
for all (temporally homogeneous) L\'evy processes.

\begin{assumption}\label{7.1}
There exist a measure $\lambda$ on $X$, a kernel $N$ on $X$, $ M_N\ge 3$, and $C_N\ge 1$ such
that, for all $x\in X$ and $0<r<R_0$, 
\begin{equation}\label{ike-wat}
C_N\inv   \vx^{B(x,r)^c}\le \int G_{B(x,r)}(x,z)N(z,\cdot)\,d\lambda(z)\le C_N \vx^{B(x,r)^c}      \on B(x,M_N r)^c.
\end{equation} 
\end{assumption}

\begin{proposition}\label{suff-i}
Suppose that  there exist   $C\ge 1$, $a\ge 3$ and $R>0$ such that, for all $x\in X$, $0<r<R$ and $y\in B(x,r)$,  
\begin{equation}\label{Nxy}
                        N(x,\cdot)\le C N(y,\cdot) \on B(y,ar)^c
\end{equation} 
 and 
\begin{equation}\label{int-xy}
  \int_{B(x,r)} g(\rho(x,z))\,d\lambda(z)\le C\int_{B(y,2r)}   g(\rho(y,z))\,d\lambda(z).
\end{equation} 
Then Assumption \ref{final-ass} is satisfied.
\end{proposition}

\begin{proof}    Let $M\ge M_N\vee (a+2)$ such that Lemma \ref{GGB} holds with $\a:=1/M$.    
Let $x\in X$, $0<r<\a (R\wedge R_0)$, $y\in B(x,r)$, and let  $E$ be a~Borel measurable set in~$B(y,M r)^c$. 
If $z\in B(y,2 r)$, then $E \subset  B(z, ar)^c$ and, by (\ref{Nxy}),
\[
      C \inv N(z,E)\le N(y,E)\le C N(z,E).
\]
Since $B(x,r)\subset B(y,2r)$, we obtain that 
\begin{eqnarray*} 
 \vx^{B(x,r)^c}(E)&\le& C_N \int G_{B(x,r)}(x,z) N(z,E)\,d\lambda(z)\\
&\le& cCC_N N(y,E)\int_{B(x,r)} g(\rho(x,z))  \,d\lambda(z)\\
&\le &c C^2C_N  N(y,E) \int_{B(y,2r)} g(\rho(y,z))  \,d\lambda(z)\\
&\le &2c^2 C^3C_N \int  G_{B(y,Mr)}(y,z) N(z,E) \,d\lambda(z) \\
&\le&  2c^2 C^3 C_N^2 \, \vy^{B(y,Mr)^c}(E). 
\end{eqnarray*} 
Thus Assumption \ref{jump} holds taking $R_1:=R\wedge R_0$. 
\end{proof} 

For simplicity, let us now assume that $X=\reald$, $\rho(x,y)=|x-y|$, 
the measure $\lambda$ in Assumption \ref{7.1} is Lebesgue measure
(a case, where clearly (\ref{int-xy}) holds) and that there exists a constant $C_G\ge 1$
such that
\begin{equation}\label{lambda-g}
      G\lambda_{B(x,r)}\le C_G g(r), \qquad \mbox{$ x\in X$, $0<r<R_0$}.
\end{equation} 
We might recall that (\ref{lambda-g}) implies that Assumption \ref{ex-c0} is satisfied
(see \cite[(1.14)]{H-liouville-wiener}).

\begin{proposition}\label{suff-ii}
Suppose that there exist a measure $\tilde \lambda$ on $\reald$, a function $n\colon \reald\times \reald\to [0,\infty)$ 
and $C\ge 1$, $a\ge 3$ such that $N(y,\cdot)=n(y,\cdot)\tilde\lambda$ for every  $y\in X$ and,
for all $x\in X$, $0<r<R_0$,  $y\in B(x,r)$ and $\tilde z\in B(x, ar)^c$,
\begin{equation}\label{decreasing}
            n(x,\tilde z)\le C n(y,\tilde z) \qquad\mbox{ provided }  |x-\tilde z|\ge |y-\tilde z|.
\end{equation} 
Then Assumption \ref{final-ass} is satisfied. 
\end{proposition}

\begin{proof}[Proof {\rm (cf.\ the proof of {\cite[Proposition 6]{grzywny}})}]  
Again,  let $M\ge M_N\vee (a+2)$ such that Lemma \ref{GGB} holds with $\a:=1/M$,    
let $x\in X$, $0<r<\a (R\wedge R_0)$, $y\in B(x,r)$, and let  $E$ be a~Borel measurable set in~$B(y,M r)^c$. 
 By~(\ref{ike-wat}),
\[
\vx^{B(x,r)^c}(E)\le cC_N \int_{B(x,r)}\int_E
g(|x-z|)n(z,\tilde z)\,d\tilde \lambda(\tilde z)\,d\lambda(z).
\]
Similarly, since $B(x,r)\subset B(y,2r)$ and $|y-z|\le 2r$ for every $z\in B(x,r)$, 
\begin{eqnarray*} 
\vy^{B(y,Mr)^c}(E) &\ge& C_N\inv \int G_{B(y,M r)}(y,z)N(z,E)\,d\lambda(z)\\
     &\ge&  (2C_N) \inv  \int_{B(x,r)} G(y,z) N(z,E)\,d\lambda(z)\\
&\ge& (2cc_DC_N)\inv  g(r)\int_{B(x,r)}\int_E n(z,\tilde z)\,d\tilde
     \lambda(\tilde z) \, d\lambda(z).
\end{eqnarray*} 
Hence it will be sufficient to show that, for every $\tilde z\in B(y,Mr)^c$,
\begin{equation}\label{suff-tilde}
\int_{B(x,r)} g(|x-z|)n(z,\tilde z)\,d\lambda(z) \le C' g(r) \int_{B(x,r)}
n(z,\tilde z) \, d\lambda(z)
\end{equation} 
(with some constant $C'>0$). So let $\tilde z\in B(y,Mr)^c$.

Let $B:=B(x,r/2)$. Since $g(|x-z|)\le g(r/2)\le c_D g(r)$ for every  $z\in B^c$,  
\[ 
\int_{B(x,r)\setminus B} g(|x-z|)n(z,\tilde z)\,d\lambda(z) \le  c_D g(r)\int_{B(x,r)}
n(z,\tilde z) \, d\lambda(z).
\]
Moreover, let
\[
          x':=x+\frac 3 4 \, \frac {\tilde z-x}{|\tilde z-x|}\, r \und  B':=B(x',r/4),
\] 
so that $B'\subset B(x,r)\setminus B$. 
If $z\in B$ and $z'\in B'$, then $|z-\tilde z|\ge |z'-\tilde z|$, 
and therefore, by (\ref{decreasing}),
\[
n(z,\tilde z)\le \frac C{ \lambda(B')} \int_{B'} n(z',\tilde z)\, d  \lambda(z')=\frac {2^d C}{ \lambda(B)} \int_{B'} n(z',\tilde z)\, d  \lambda(z').
\]
Hence
\begin{multline*} 
 \int_B g(|x-z|)n(z,\tilde z)\,d\lambda(z)\\
 \le  2^d C\left(\int_{B'} n(z',\tilde z)\, d  \lambda(z')\right)\cdot
\left(\frac 1{\lambda(B)}\int_B g(|x-z|)\,d\lambda(z)\right),
\end{multline*} 
where
\[\frac1{\lambda(B )} 
\int_Bg(|x-z|)\,d\lambda(z)\le cG\lambda_B(x)\le  cC_G g(r/2) \le cc_DC_G g(r). 
\]
Thus (\ref{suff-tilde}) holds with $C':=c_D(1+2^dcCC_G)$. 
\end{proof} 

If $y\in B(x,r)$  and $\tilde z\in B(x,2r)^c$, then $|x-\tilde z| \le 2|x-y|+2|y-\tilde z|-|x-\tilde z|<2|y-\tilde z|$.
Hence we have the following result.

\begin{corollary}\label{iso}
Suppose that there exists a measure $\tilde \lambda$ on $\reald$ 
   such that $N(y,\cdot)=n(y,\cdot)\tilde\lambda$, $y\in X$, where
 $n(x,y)\approx n_0(|x-y|)$, and  that there exists  $C_0\ge 1$ such that   
\begin{equation}\label{weak-decreasing}
                         n_0(s) \le C_0  n_0(r), \qquad\mbox{ whenever } 0<r<s<2r.
\end{equation} 
Then Assumption \ref{final-ass} holds.
\end{corollary} 

Thus  rather general L\'evy processes may serve as examples for our
approach (see, for example, \cite{grzywny, kim-mimica,
mimica-harnack,  mimica-harmonic, 
rao-song-vondracek, sikic-song-vondracek}).

\section{Harnack inequalities under intrinsic assumptions}\label{section-local}

In this section, 
let us assume that $(X,\W)$ is a balayage space, $1\in \W$, and that we have a Borel measurable function
 $G\colon X\times X\to (0,\infty]$ such that  $G=\infty$ on the diagonal
and the following holds:
\begin{itemize} 
\item [\rm (i)] 
   For every $y\in X$,  $ G(\cdot,y)$ is a potential which is harmonic on $X\setminus \{y\}$.   
\item [\rm (ii)] 
For every potential $p$ on $X$, there exists a~measure $\mu$ on $X$    such that 
\[
p=G\mu:=\int G(\cdot,y)\,d\mu(y).
\]
\end{itemize} 
Moreover, we assume             that there is a~function             
  $w \in \W\cap \C(X)$, $0<w\le 1$,  such that  each function $G(\cdot,x)/w$, $x\in X$,  is  bounded at infinity
and $G$ has the \emph{$(w,w)$-triangle property} (see \cite{H-uniform}), 
that is, for some constant  $C_w >1$, the function  
\[
\tilde G\colon (x,y)\mapsto G(x,y)/(w(x)w(y))
\] 
 satisfies
\begin{equation}\label{ww-triangle}
   \tilde G(x,z)\wedge \tilde G(y,z)\le C_w  \tilde G(x,y), \qquad x,y,z\in X.
\end{equation} 

For $x\in X$ and $r>0$, we define open neighborhoods $V(x,r)$ of $x$ by
\begin{equation}\label{vxr}
              V(x,r):=\{  G(\cdot,x)>1/r\}.
\end{equation} 
We intend to prove the following result (where $\tilde \H^+(U)$ has the same meaning as in Section 1).

\begin{theorem}\label{harnack-general}
Let $U$ be an      
open set which is covered by open sets $V$ having the following property: There are real numbers
$R_1\in (0,\infty]$,  $C, c_J\in (1,\infty)$ 
and $\a\in (0,1)$ {\rm(}which may depend on $V${\rm)} such that, for all $0<r<R_1$ and $x\in V$, 
\begin{equation}\label{RVr}
     R_1^{V(x,r) }\ge C\inv   G(\cdot,x)  r \on V(x,r)^c
\end{equation} 
and, for all $y\in V(x,\a r)$,
\begin{equation}\label{v-jump}
\vx^{V(x, \a r)^c}\le c_J \vy^{V(y,r)^c} \on V(y,r)^c.
\end{equation} 
 
Then
$\tilde \H^+(U)=\H^+(U)$ and, for every $x\in U$, there exists a compact neighborhood~$L$ of~$x$ in~$U$
and a constant $K\ge 1$ such that 
\begin{equation}\label{har-allgemein}
\sup h(L) \le K \inf h(L)\qquad \mbox{ for every }h\in \H^+(U).
\end{equation} 
\end{theorem} 

\begin{remarks}{\rm
1. Of course, similar properties as in Section \ref{suff-ass},    locally in~$x$,
will be sufficient for (\ref{v-jump}). 

2. If $U$ is arcwise connected, then  standard arguments show that, for every compact $L$  
in $U$, there exists $K\ge 1$ such that (\ref{har-allgemein}) holds.
}
\end{remarks}

For a proof of Theorem \ref{harnack-general}, let us first recall that, defining
\begin{equation}\label{tilde-W}
  \widetilde \W:=\{\frac uw\colon u\in \W\},
\end{equation} 
we have a balayage space $(X,\widetilde \W)$ such that $1\in \tilde\W$ and, for every positive function 
$f\ge 0$ on $X$,
\begin{equation}\label{red-f}
                                                  \tilde R_f:=\inf \{\tilde v\in \widetilde\W\colon \tilde v\ge  f\}
              = \frac 1w\, R_{fw}.
\end{equation} 
In particular, for all $x\in X$ and $A\subset X$, the reduced measure $\tilde \ve_x^A$ with respect to~$(X,\widetilde \W)$ is
\begin{equation}\label{bal-tilde}
                                   \tilde \ve_x^A=\frac w{w(x)} \, \vx^A.
\end{equation} 
Therefore a function $h$ is harmonic on $U$ with respect to $(X,\W)$ if and only if the function $h/w$
is harmonic with respect to $(X,\widetilde \W)$. 
Moreover, it is easily verified that $\tilde G$ is a~Green function for~$(X,\tilde\W)$:
  a~function $p$ on $X$ is a potential for $(X,\W)$ if and only if $p/w$
is a potential for~$(X,\widetilde\W)$. Clearly $(1/w)G\mu = \tilde G(w\mu)$ for every measure~$\mu$ on~$X$.

Since $\tilde G=\infty$ on the diagonal, (\ref{ww-triangle}) implies that $\tilde G(y,x)\le C \tilde G(x,y)$
and 
$                              \tilde \rho(x,y):= \tilde G(x,y)\inv + \tilde G(y,x)\inv$, $x,y\in X$,
defines a quasi-metric on $X$ which is equivalent to $\tilde G\inv$. By  \cite[Proposition 14.5]{heinonen}
(see also \cite[Proposition~6.1]{H-liouville-wiener}), there exists a metric~$d$ on~$X$
and $\g>0$ such that $\tilde \rho \approx d^{\g}$. So  there exists  $c\ge 1$ with
\begin{equation}\label{Gd}
          c\inv d^{-\g}\le \tilde G\le c d^{-\g}.
\end{equation} 
For $ x\in X$ and  $  r>0$, let 
\begin{equation}\label{bxr}
            B(x,r):=\{y\in X\colon d(y,x)<r\}. 
\end{equation} 
Clearly,   
\begin{equation}\label{BV}
B(x,r)\supset \bigl\{\tilde G(\cdot,x)>cr^{-\g}\bigr\}\supset   V(x,c\inv r^{\g}) .
\end{equation} 
Further, if $V$ is a relatively compact  neighborhood of $x$, then, by assumption, 
$  G(\cdot,x)/w$ is bounded on $X\setminus V$;
so there exists $M>0$ such that $\{\tilde G(\cdot,x)>M\}\subset V$, and hence
$B(x,(Mc)^{-1/\g}) \subset V$. 
Therefore $d$ is a metric for the topology of $X$.

Thus Assumption \ref{main-ass} is satisfied for $(X,\widetilde \W) $ and $\tilde G$ taking 
\[ 
\rho:=d,  \quad g(r):=r^{-\g}, \quad  R_0:=\infty, \quad c_D:=2^\g,\quad \eta_0:=\a_0^\g.
\]

\begin{proof}[Proof of Theorem \ref{harnack-general}]  
Let us fix $x_0\in U$, and let $V$ be a~relatively compact  open neighborhood of~$x_0$ in $U$ 
(with corresponding $R_1,c_J,\a$) having the properties
stated in Theorem~\ref{harnack-general}. 
We choose $0<R\le R_1\wedge (c R_1)^{1/\g}$ 
such that $B(x_0,2R)$ is a~proper subset of $V$, 
and define $a,\b\in(0,1)$ by 
\begin{equation}\label{abb}
a:=\inf w(V)  \und  \b:=(a/c)^{2/\g}.  
\end{equation} 

Let $x\in B(x_0,R)$, $0<r<R$, $B:=B(x,r)$, $\tilde r:=c\inv r^\g$. Then
\begin{equation}\label{BVB}
       B(x,  \b r)\subset   V(x,\tilde r)\subset B\subset V.
\end{equation} 
Indeed, of course, $B\subset V$, and, by (\ref{BV}), $ V(x,\tilde r)\subset B$.
And $B(x,\b r)$ is contained in $V(x,\tilde r)$, since, for every $y\in B(x,\b r)\subset B\subset V$,
\[
G(y,x)\ge a^2\tilde G(y,x)\ge  a^2 c\inv d(x,y)^{-\g}>  a^2c\inv  (\b r)^{-\g}=1/\tilde r.    
\]
Since $w\le 1$ and $\tilde r<R_1$, we see, by (\ref{red-f}), (\ref{BVB}), and (\ref{RVr}), that
\begin{equation}\label{Rww}
 \tilde R_1^{B} =\frac 1w\,R_{w}^{B}   \ge  a R_1^{V(x,\tilde r)} \ge   a C\inv  G(\cdot, x) \tilde r 
\on V(x,\tilde r)^c. 
\end{equation} 
In particular, fixing $z\in V\setminus B(x_0,2R)$, we have
\[
\tilde R_1^{B} (z) \ge a^3 C\inv \tilde G(z,x) \tilde r\ge a^3 (c^2C)\inv   g(d(z,x)) / g(r).
\]
Given  $\ve>0$, there is
   $0<r'<r$ such that $B':=B(x,r')$ satisfies 
$\tilde R_1^{B'}(z)+\ve >\tilde R_1^{B}(z)$, where (denoting the capacity  of $B$ with respect to $\tilde G$ by $\tkap\,B$)
\[
     \tilde  R_1^{B'} (z) =\int \tilde G(z,y)\,d\tilde \mu_{B'} (y)\le  c g(d(z,x)/2)  \|\tilde \mu_{B'}\|\le cc_D g(d(z,x))\tkap\, B,
\] 
since $d(z,x)/2\le d(z,x) -r< d(z,\cdot)$ on $B$  (cf.\ the proof of   \cite[Proposition 1.10,b]{H-liouville-wiener}). 
So 
\[
       \tkap\, B\ge a^3 (c^3c_DC)\inv  g(r)\inv. 
\]

Next, let   $y\in B(x,\a \b r) $. By  (\ref{it-bal}), (\ref{v-jump}),  and   (\ref{BVB}) (applied to $\b r$ and $r$), 
\[
  \ve_x^{B(x,\a \b r)^c}\le         \ve_x^{V(x,\a \tilde r)^c}\le   c_J \ve_y^{V(x,\tilde r)^c}\le c_J \vy^{B(x,r)^c} \on B(x,r)^c.
\]
Hence,  by (\ref{bal-tilde}),  
\[
\tilde \ve_x^{B(x,\a \b r)^c}=\frac{w}{w(x)} \,\ve_x^{B(x,\a \b r)^c}\le a\inv c_J\,\frac {w}{w(y)} \, \vy^{B(x,r)^c} 
= a\inv c_J\tilde \ve_y^{B(x,r)^c} \on B(x,r)^c.
\]

Thus, by Remark \ref{ess-local}, we conclude that there exist    constants $\tilde \a\in (0,1/4)$ and $\tilde K\ge 1$ such that, 
 for every function $\tilde h\ge 0$ which is harmonic on $U$ with respect to~$(X, \widetilde \W)$,
\[
\sup \tilde h(\ov B(x_0,\tilde a R))\le \tilde K\inf \tilde h(\ov B(x_0,\tilde \a R) ).
\]
Finally, if $h\in \H^+(U)$, then $h/w$ is harmonic on $U$ with respect to~$(X, \widetilde \W)$,
and thus
\[
\sup h(\ov B(x_0,\tilde \a R))\le a\inv \tilde K\inf   h(\ov B(x_0,\tilde \a R) ).
\]
Of course, we  obtain as well that $\tilde \H^+(U)=\H^+(U)$.
\end{proof}

\bibliographystyle{plain} %  {alpha}
%\bibliography{lit_bank.bib}
\def\cprime{$'$} \def\cprime{$'$}

{\small \noindent 
Wolfhard Hansen,
Fakult\"at f\"ur Mathematik,
Universit\"at Bielefeld,
33501 Bielefeld, Germany, e-mail:
 hansen$@$math.uni-bielefeld.de}\\
{\small \noindent Ivan Netuka,
Charles University,
Faculty of Mathematics and Physics,
Mathematical Institute,
 Sokolovsk\'a 83,
 186 75 Praha 8, Czech Republic, email:
netuka@karlin.mff.cuni.cz}

\end{document}